\documentclass{amsart}

\usepackage[alphabetic]{amsrefs}

\usepackage[all]{xy}
\usepackage{tikzsymbols}
\usepackage{syntonly}
\usepackage{enumerate}
\usepackage{hyperref}
\usepackage{amsfonts}
\usepackage{amssymb}
\usepackage{amsmath}
\usepackage{amsthm}
\usepackage{enumitem}
\usepackage{tikz}
\usepackage{tikz-cd}
\usepackage{graphics}
\usepackage{graphicx}
\usepackage{color}
\usepackage{comment}

\usepackage{ifsym}

\usepackage{lineno}

\newtheorem{theorem}{Theorem}
\newtheorem{corollary}[theorem]{Corollary}

\newtheorem{definition}[theorem]{Definition}
\newtheorem{lemma}[theorem]{Lemma}

\newtheorem{nonGlobalClaim}{Claim}[theorem]  

\newtheorem{observation}[theorem]{Observation}
\newtheorem*{question}{Question}

\newtheorem{fact}[theorem]{Fact}

\newcommand{\CH}{{\rm CH}}
\newcommand{\GCH}{{\rm GCH}}

\newcommand{\MA}{{\rm MA}}

\renewcommand{\P}{{\mathbb P}}

\newcommand{\Q}{{\mathbb Q}}
\newcommand{\R}{{\mathbb R}}

\newcommand{\Add}{\mathop{\rm Add}}

\renewcommand{\Col}{\mathop{\rm Col}}
\newcommand{\Code}{\mathop{\rm Code}}

\newcommand{\forces}{\Vdash}
\newcommand{\forced}{\Vdash}

\newcommand{\restrict}{\upharpoonright}

\newcommand{\<}{\langle}
\renewcommand{\>}{\rangle}
\newcommand{\elemsub}{\prec}

\newcommand{\st}{:}

\newcommand{\supp}{\mathop{\rm supp}}

\newcommand{\dom}{\mathop{\rm dom}}

\newcommand{\range}{\mathop{\rm range}}

\newcommand{\ot}{\mathop{\rm ot}\nolimits}
\newcommand{\cf}{\mathop{\rm cf}}
\newcommand{\cof}{\mathop{\rm cof}}

\newcommand{\id}{\mathop{\rm id}}

\newcommand{\crit}{\mathop{\rm crit}}

\renewcommand{\and}{\mathop{\&}}

\renewcommand{\diamond}{\diamondsuit}

\author[Brent Cody]{Brent Cody}
\address[Brent Cody]{ 
Virginia Commonwealth University,
Department of Mathematics and Applied Mathematics,
1015 Floyd Avenue, PO Box 842014, Richmond, Virginia 23284
} 
\email[B. ~Cody]{bmcody@vcu.edu} 
\urladdr{http://www.people.vcu.edu/~bmcody/}

\author[Monroe Eskew]{Monroe Eskew}
\address[Monroe Eskew]{ 
Virginia Commonwealth University,
Department of Mathematics and Applied Mathematics,
1015 Floyd Avenue, PO Box 842014, Richmond, Virginia 23284
} 
\email[M. ~Eskew]{mbeskew@vcu.edu}

\thanks{The authors would like to thank Sean Cox for suggesting this topic as well as several fruitful approaches {\Large $\Wintertree$ $\Coffeecup$ $\Wintertree$}}

\begin{document}

\title{Rigid ideals}

\maketitle


\begin{abstract}
An ideal $I$ on a cardinal $\kappa$ is called \emph{rigid} if all automorphisms of $P(\kappa)/I$ are trivial. An ideal is called \emph{$\mu$-minimal} if whenever $G\subseteq P(\kappa)/I$ is generic and $X\in P(\mu)^{V[G]}\setminus V$, it follows that $V[X]=V[G]$. We prove that the existence of a rigid saturated $\mu$-minimal ideal on $\mu^+$, where $\mu$ is a regular cardinal, is consistent relative to the existence of large cardinals.  The existence of such an ideal implies that $\GCH$ fails. However, we show that the existence of a rigid saturated ideal on $\mu^+$, where $\mu$ is an \emph{uncountable} regular cardinal, is consistent with $\GCH$ relative to the existence of an almost-huge cardinal. Addressing the case $\mu=\omega$, we show that the existence of a rigid \emph{presaturated} ideal on $\omega_1$ is consistent with $\CH$ relative to the existence of an almost-huge cardinal. The existence of a \emph{precipitous} rigid ideal on $\mu^+$ where $\mu$ is an uncountable regular cardinal is equiconsistent with the existence of a measurable cardinal.
\end{abstract}

\section{Introduction}\label{section_introduction}

An ideal $I$ on a cardinal $\kappa$ is said to be \emph{rigid} if all automorphisms of the boolean algebra $\mathcal{P}(\kappa)/I$ are trivial. Woodin proved \cite{Woodin} that if $\MA_{\omega_1}$ holds, then every saturated ideal on $\omega_1$ is rigid. Larson \cite{MR1955241} showed that we do not need the whole of Martin's Axiom to obtain the same conclusion; more specifically, Larson proved that if a certain cardinal characteristic of the continuum is greater than $\omega_1$, then every saturated ideal on $\omega_1$ is rigid. It is shown in \cite[Theorem 18]{MR924672}, that in models of $\MA_{\omega_1}$, every saturated ideal $I$ on $\omega_1$ has an additional property: if $G\subseteq\mathcal{P}(\omega_1)/I$ is generic and $r\in V[G]\setminus V$ is a real, then $V[r]=V[G]$. Given a cardinal $\mu$, we say that a poset $\P$ is \emph{$\mu$-minimal} if whenever $G\subseteq \P$ is generic and $X\in \mathcal{P}(\mu)^{V[G]}\setminus V$, it follows that $V[X]=V[G]$. When we say that an ideal $I\subseteq\mathcal{P}(\kappa)$ is $\mu$-minimal, we mean that the poset $\mathcal{P}(\kappa)/I$ is $\mu$-minimal. Thus, under $\MA_{\omega_1}$, every saturated ideal on $\omega_1$ is rigid and $\omega$-minimal.

In this article we extend the above results on rigidity and minimality properties of ideals. We first note the following easy generalization:
\begin{observation}
If $\kappa>\omega_1$ is a regular cardinal carrying a saturated ideal $I$, then there is a c.c.c.\ forcing $\mathbb P$ such that $\Vdash_{\mathbb P}$ ``The ideal generated by $I$ is rigid, saturated, and $\omega$-minimal.''
\end{observation}
To prove this, we let $\mathbb P$ be the Solovay-Tennenbaum forcing \cite{MR0294139} to obtain MA$_\kappa$.  An easy application of Corollary \ref{theorem_baumgartner_taylor} below shows that the saturation of $I$ is preserved.  The arguments for rigidity and $\omega$-minimality are identical to those of \cite{MR1955241} and \cite{MR924672}.  We do not know if such ideals will be $\mu$-minimal for other $\mu < \kappa$. Nonetheless, the following theorem shows that the existence of a rigid, saturated, $\mu$-minimal ideal on $\mu^+$, where $\mu$ is an uncountable regular cardinal, is consistent relative to large cardinals.

\begin{theorem}\label{theorem_rigid_saturated_minimal}
Suppose $\GCH$ holds and $I$ is a normal saturated ideal on $\mu^+$ where $\mu$ is a regular uncountable cardinal. Then there is a ${<}\mu$-distributive forcing extension in which the ideal generated by $I$ is rigid, saturated and $\mu$-minimal.
\end{theorem} 

The proof of Theorem \ref{theorem_rigid_saturated_minimal} given below uses methods of Larson \cite{MR1955241} mentioned above which involve exploiting the fact that one may force a certain cardinal characteristic to be large, and thus in the forcing extension, $2^\mu > \mu^+$.

Notice that in all the models with rigid saturated ideals mentioned thus far, $\GCH$ fails. Indeed, as we show in Section \ref{section_rigid_minimal}, $\GCH$ implies that for every regular cardinal $\mu$, there does not exist a $\mu$-minimal presaturated ideal on $\mu^+$. It is natural to wonder whether or not the situation changes if we remove the minimality requirement: is the existence of a rigid saturated ideal on some successor cardinal consistent with $\GCH$? We will show that the existence of a rigid saturated ideal on $\mu^+$, where $\mu$ is an \emph{uncountable} regular cardinal, is consistent with $\GCH$, relative to the existence of an almost-huge cardinal.

\begin{theorem}\label{theorem_rigid_saturated}
Suppose that $\kappa$ is an almost-huge cardinal and $\mu<\kappa$ is an uncountable regular cardinal. Then there is a ${<}\mu$-distributive forcing extension in which there is a rigid saturated ideal on $\mu^+$ and $\GCH$ holds.
\end{theorem}

Notice that Theorem \ref{theorem_rigid_saturated} fails to address the case of ideals on $\omega_1$. We will show that it is consistent relative to the existence of an almost-huge cardinal that $\omega_1$ carries a rigid \emph{presaturated} ideal while $\GCH$ holds.

\begin{theorem}\label{theorem_rigid_presaturated}
Suppose that $\kappa$ is an almost-huge cardinal and $\mu<\kappa$ is a regular cardinal. Then there is a ${<}\mu$-distributive forcing extension in which there is a rigid presaturated ideal on $\mu^+$ and $\GCH$ holds.
\end{theorem}
Our proofs of Theorem \ref{theorem_rigid_saturated} and Theorem \ref{theorem_rigid_presaturated} will involve using a variation of a coding forcing introduced by Friedman and Magidor in \cite{MR2548481}, which they used to control the number of normal measures carried by a measurable cardinal.\footnote{A variation of Friedman and Magidor's coding forcing was used by Ben-Neria \cite{MR3397347} to show that \emph{any} well founded order can be realized as the Mitchell order $\lhd(\kappa)$ on a measurable cardinal $\kappa$.} Assuming $\kappa$ is an almost-huge cardinal and $\mu<\kappa$ is regular, we will use a coding forcing to define a forcing $\P$ such that if $G\subseteq\P$ is generic over $V$, then in $V[G]$ we have $\kappa=\mu^+$, there is a saturated ideal $I$ on $\kappa$, and forcing with $\mathcal{P}(\kappa)/I$ over $V[G]$ produces an extension $V[G*H]$ in which $H$ is the unique generic filter for $\mathcal{P}(\kappa)/I$ over $V[G]$. Hence the ideal $I\in V[G]$ is rigid. 

The situation in the present article differs from the Friedman and Magidor results \cite{MR2548481} in the following interesting way. One may force a measurable cardinal $\kappa$ to cary exactly two measures; whereas if $I$ is a normal ideal on a regular cardinal $\kappa$ and the boolean algebra $P(\kappa)/I$ has a nontrivial automorphism (hence there are at least two generic filters in $V^{P(\kappa)/I}$ for $P(\kappa)/I$), then $P(\kappa)/I$ must have $2^\kappa$ nontrivial automorphism. See Section \ref{section_questions} for more details and an open question.

We prove Theorem \ref{theorem_rigid_saturated} and Theorem \ref{theorem_rigid_presaturated} below with the proof of Theorem \ref{theorem_rigid_presaturated} coming before that of Theorem \ref{theorem_rigid_saturated}, because the forcing construction for Theorem \ref{theorem_rigid_saturated} is more technical. Our proof of Theorem \ref{theorem_rigid_saturated} will be split into two cases: first we prove Theorem \ref{theorem_rigid_saturated} for $\mu$ not the successor of a singular cardinal (see Theorem \ref{theorem_rigid_saturated_mu_not_succ_sing}), then we prove the remaining case in Section \ref{section_singular}.

We also show that a measurable cardinal will suffice to obtain a model with a \emph{precipitous} rigid ideal on $\mu^+$ where $\mu$ is a regular uncountable cardinal.

\begin{theorem}\label{theorem_rigid_precipitous}
Suppose $\kappa$ is a measurable cardinal and $\mu<\kappa$ is an uncountable regular cardinal. Then there is a forcing extension in which there is a rigid precipitous ideal on $\mu^+$ and $\GCH$ holds.
\end{theorem}

A similar result is not possible for $\omega_1$.  The existence of a presaturated ideal on $\omega_1$ is known to be equiconsistent with a Woodin cardinal~\cite{Woodin}.  If $\GCH$ holds and $I$ is a precipitous but not presaturated ideal on $\omega_1$, then $\mathcal P(\omega_1)/I$ is forcing-equilvalent to $\Col(\omega,\omega_2)$, which never has unique generics.

\section{Preliminaries}\label{section_preliminaries}

Let us review some absorption properties of collapse forcings. Let $\mu$ be a regular cardinal. If $\P$ is a  ${<}\mu$-closed separative forcing, then for sufficiently large $\kappa$ it follows that there is a regular embedding $\P\to\Col(\mu,\kappa)$ and we say that $\Col(\mu,\kappa)$ absorbs $\P$. If $\kappa>\mu$ is an inaccessible cardinal and $\P$ is any ${<}\mu$-closed separative forcing with $|\P|<\kappa$ then there is a regular embedding $\P\to \Col(\mu,{<}\kappa)$. See \cite[Section 14]{MR2768691} for more details.

In order to force the existence of a saturated ideal on $\omega_1$ starting with a model containing a huge cardinal $\kappa$, Kunen \cite{MR495118} defined a forcing iteration $\mathbb{K}$ of length $\kappa$ which is $\kappa$-c.c.\ and highly universal in the sense that many posets regularly embed into $\mathbb{K}$, including many posets of size $\kappa$. We refer the reader to \cite{CoxKunenCollapse} for additional background on universal collapsing forcing. Let us review the definition of a slight variation of Kunen's universal collapse which was used by Magidor (see \cite{MR2768692}), as well as some of its properties that will be relevant for our proofs of Theorem \ref{theorem_rigid_saturated} and Theorem \ref{theorem_rigid_presaturated}. 

\begin{definition}\label{definition_kunen_iteration}
Suppose $\mu<\kappa$ are regular cardinals. Let $\P=\P_\kappa$ be a ${<}\mu$-support iteration $\<(\P_\alpha,\dot{\Q}_\beta)\st\alpha\leq\kappa, \beta<\kappa\>$ such that 
\begin{enumerate}
\item $\P_0=\Col(\mu,{<}\kappa)$
\item If $\P_\alpha\cap V_\alpha$ is a regular suborder of $\P_\alpha$ and $\P_\alpha\cap V_\alpha$ is $\alpha$-c.c., we say that $\alpha$ is an \emph{active} stage in the iteration, and let $\dot{\Q}_\alpha$ be a $\P_\alpha\cap V_\alpha$-name for $\Col(\alpha,{<}\kappa)^{V^{\P_\alpha\cap V_\alpha}}$. \footnote{In Kunen's original definition, the Silver collapse is used at such stages $\alpha$ so that certain master conditions exist.}
\end{enumerate}
\end{definition}
In the proofs of Theorem \ref{theorem_rigid_saturated} and Theorem \ref{theorem_rigid_presaturated} below we will need the following properties of this iteration.

\begin{lemma}
Suppose $\kappa$ is almost-huge and $\P=\P_\kappa$ is the iteration defined above. The following properties hold.
\begin{enumerate}
\item $\P_\kappa$ is ${<}\mu$-distributive and forces $\kappa=\mu^+$;
\item $\P_\kappa\subseteq V_\kappa$;
\item $\P_\kappa$ is $\kappa$-c.c.;
\item for each inaccessible $\gamma<\kappa$ there is a regular embedding $e_{\gamma,\kappa}:\P_\gamma*\Col(\gamma,<\kappa)\to \P_\kappa$ and 
\item \label{absorb_collapse} whenever $G*H$ is generic for $\P_\kappa*\Col(\kappa,{<}\lambda)$ over $V$, there is a regular embedding $e:\Col(\mu,\kappa)\to\P_\lambda/(G*H)$.
\end{enumerate} 
\end{lemma}
In the proof of Theorem \ref{theorem_rigid_saturated} below we will use a different variation of Kunen's universal collapse. The fact that the chain condition holds for this variation will follow from a result of Cox.
\begin{theorem}[\cite{CoxKunenCollapse}, Theorem 39]\label{theorem_cox_kunen_collapse}
Suppose $\kappa$ is weakly compact and $\<(\P_\alpha,\dot{\Q}_\alpha)\st \alpha\leq\kappa,\beta<\kappa\>$ is a ``Kunen-style'' universal iteration (see \cite[Definition 34]{CoxKunenCollapse}). Suppose
\begin{enumerate}
\item direct limits are taken at all inaccessible $\gamma\leq\kappa$,
\item for every active $\alpha<\kappa$ we have $\forced_{V_\alpha\cap \P_\alpha}\dot{\Q}_\alpha\subset V_\kappa[\dot{g}_\alpha]$ and
\item each $\dot{\Q}_\alpha$ is forced by $V_\alpha\cap\P_\alpha$ to be $\kappa$-c.c.
\end{enumerate}
Then $\P_\kappa\subseteq V_\kappa$ is ``layered'' on some stationary subset of 
\[\Gamma:=\{W\in P_\kappa(V_\kappa)\st \text{$W=V_\gamma$ for some inaccessible $\gamma<\kappa$}\}.\]
In particular, $\P_\kappa$ is $\kappa$-Knaster.
\end{theorem}

Generic large cardinal properties have been extensively studied \cite{MR2768692}, and have many applications in the form of consistency results at successor cardinals. Suppose $j:V\to M\subseteq V[G]$ is a generic elementary embedding with critical point $\kappa$, where $G$ is generic over $V$ for a forcing $\P$. One fundamental feature of many applications of generic embeddings is that, in certain situations, the forcing $\P$ which adds the embedding is forcing equivalent to $\mathcal{P}(\kappa)/I$ for a particular naturally defined ideal $I\in V$.  Several notions about these kinds of ideals are:

\begin{definition}
If $\kappa$ is a regular cardinal, and $I$ is a $\kappa$-complete ideal on $\kappa$, then we say:
\begin{enumerate}
\item $I$ is precipitous if whenever $G \subseteq \mathcal P(\kappa)/I$ is generic over $V$, then $V^\kappa/G$ is well-founded.
\item $I$ is saturated if $\mathcal P(\kappa)/I$ has the $\kappa^+$-c.c.
\item $I$ is presaturated if forcing with $\mathcal P(\kappa)/I$ preserves $\kappa^+$.
\end{enumerate}
\end{definition}

\begin{fact}[See \cite{MR2768692}]
If $I$ is a $\kappa$-complete presaturated ideal on $\kappa$ and $2^\kappa = \kappa^+$, then
\begin{enumerate}
\item $I$ is precipitous.
\item If $G \subseteq \mathcal P(\kappa)/I$ is generic $M \cong V^\kappa/G$ is transitive, then $M^\kappa \cap V[G] \subseteq M$.
\end{enumerate}
\end{fact}

Foreman showed that many of these applications involving the correspondence between forcings which add generic embeddings and naturally defined ideals in the ground model, can be unified, and viewed as easy consequences of a very general theorem he called the Duality Theorem. Here we state two weak versions of Foreman's Duality Theorem which we will use in our proofs of Theorem \ref{theorem_rigid_saturated}, Theorem \ref{theorem_rigid_presaturated} and Theorem \ref{theorem_rigid_precipitous}.

\begin{theorem}[Foreman, \cite{MR3038554}]\label{theorem_duality_theorem}
Suppose $Z$ is a set and $\P$ is a forcing such that whenever $G\subseteq\P$ is generic, there is an ultrafilter $U$ on $Z$ such that $V^Z/U$ is isomorphic to a transitive class $M$. Also assume that there are functions $f_\P$, $\<f_p\>_{p\in\P}$ and $g$ such that $\forced_\P$ ``$[f_\P]_U=\P$, $(\forall p\in\P)[f_p]_U=p$ and $[g]_U=\dot{G}$.'' If $I=\{X\subseteq Z\st 1 \forced_\P [\id]_U\notin j_U(X)\}$, then there is a dense embedding $d:\mathcal{P}(Z)/I\to\mathcal{B}(\P)$.
\end{theorem}

\begin{theorem}[Foreman, \cite{MR3038554}]
\label{dualitynicecase}
Suppose $I$ is a precipitous $\kappa$-complete ideal on $Z$ and $\mathbb{P}$ is a $\kappa$-c.c. partial order.  If $\bar I$ denotes the ideal generated by $I$ in $V^{\mathbb P}$, then $\mathcal{B}( \mathbb{P} * \dot{\mathcal{P}} (Z)/ \bar I) \cong \mathcal{B}( \mathcal P(Z)/I * \dot{j(\mathbb{P})} )$.

\end{theorem}

The following result of Baumgartner and Taylor \cite{MR645330} follows immediately from Theorem \ref{dualitynicecase}:

\begin{corollary}\label{theorem_baumgartner_taylor}
Suppose $\kappa$ is a successor cardinal, $I$ is a $\kappa^+$-saturated ideal on $\kappa$, and $\mathbb P$ is a $\kappa$-c.c.\ forcing.  Then the ideal generated by $I$ in $V^{\mathbb P}$ is $\kappa^+$-saturated if and only if $\Vdash_I j(\mathbb P)$ is $\kappa^+$-c.c.
\end{corollary}

\section{Rigidity with minimal generics}\label{section_rigid_minimal}

In this section, we will prove that it is consistent for $\mu^+$, the successor of a regular uncountable cardinal, to carry a rigid saturated $\mu$-minimal ideal.  First we note the following obstruction.

\begin{observation}\label{observation_GCH_must_fail}
	If $\GCH$ holds and $\mu$ is regular, then there is no $\mu$-minimal presaturated ideal on $\mu^+$.
\end{observation}

\begin{proof}Suppose $I$ is a presaturated idea on $\mu^+$, and $j : V \to M \subseteq V[G]$ is a generic ultrapower embedding derived from $I$.  By $\GCH$ and the closure of $M$, $([\mu]^{<\mu})^V = j([\mu]^{<\mu}) = ([\mu]^{<\mu})^M = ([\mu]^{<\mu})^{V[G]}$.  In $V[G]$, $|\mathcal{P}(\Add(\mu))^V| = \mu$.  Therefore, in $V[G]$, we can recursively choose a sequence $\langle p_\alpha : \alpha < \mu \rangle \subseteq \Add(\mu)$ that generates a $V$-generic filter.  If $X \subseteq \mu$ is the subset of $\mu$ coded by this sequence, then $V[X] \not= V[G]$, since $V[X]$ has the same cardinals as $V$.
\end{proof}

This implies that any forcing used to produce an extension with a rigid saturated $\mu$-minimal ideal, must necessarily force $\GCH$ to fail.  We now prove Theorem \ref{theorem_rigid_saturated_minimal} by starting with a model of $\GCH$ in which there is a saturated ideal on $\mu^+$ where $\mu$ is a regular cardinal, forcing to control a certain cardinal characteristic and then carrying out the relevant arguments of \cite{Woodin} and \cite{MR1955241} in this context.

\begin{proof}[Proof of Theorem \ref{theorem_rigid_saturated_minimal}]

Suppose $\mathcal A$ is an antichain in $\mathcal{P}(\mu)/ \{$bounded sets$\}$.  Following \cite{KunenBook}, we define a forcing $\mathbb C_{\mathcal A}$:  Conditions are of the form $p = (s,T)$, where $s$ is a bounded subset of $\mu$, and $T$ is a subset of $\mathcal A$ of size $<\mu$.  We say $(s_1,T_1) \leq (s_0,T_0)$ when $s_0 \subseteq s_1$, $T_0 \subseteq T_1$, and for all $t \in T_0$, $s_1 \cap t = s_0 \cap t$. Clearly, if $G \subseteq  \mathbb C_{\mathcal A}$ is generic, and $s_G = \bigcup \{s : \exists T (s,T) \in G \}$, then $s_G \cap a$ is bounded in $\mu$ for all $a \in \mathcal A$.

\begin{lemma}
	Suppose $\mathbb C_{\mathcal A}$, $G$, and $s_G$ are as above, and suppose $b \in \mathcal P(\mu)^V$ is an unbounded subset of $\mu$ such that $|b \cap a| < \mu$ for all $a \in \mathcal A$.  Then $s_G \cap b$ is unbounded in $\mu$.
\end{lemma}

\begin{proof}
	Let $(s,T)$ be any condition, and let $\alpha < \mu$ be arbitrary.  Since $|b \cap t| < \mu$ for all $t \in T$, there is $\beta \geq \alpha$ such that $\beta \in b \setminus \bigcup T$.  Then $(s \cup \{\beta\},T) \leq (s,T)$, and this condition forces $\sup(s_G \cap b) \geq \alpha$.
\end{proof}

\begin{lemma}
	If $\mathbb C_{\mathcal A}$ is as above, then it is ${<}\mu$-closed and $2^{<\mu}$-centered.
\end{lemma}

\begin{proof}
If $\langle (s_\alpha,T_\alpha) : \alpha < \beta \rangle$ is a descending sequence with $\beta < \mu$, then $(\bigcup_\alpha s_\alpha,\bigcup_\alpha T_\alpha)$ is the infimum of the sequence.  If $(s,T_0)$ and $(s,T_1)$ are two conditions, then $(s,T_0 \cup T_1)$ is their infimum.
\end{proof}

Let $\P_{\kappa^+}$ be a ${<}\mu$-support forcing iteration $\langle (\mathbb P_\alpha, \dot{\mathbb Q}_\alpha) : \alpha < \kappa^+ \rangle$ satisfying the following properties.
\begin{enumerate}
\item For each $\alpha$, $\Vdash_\alpha \dot{\mathbb Q}_\alpha = \mathbb C_{\mathcal A}$ for some antichain $\mathcal A \subseteq \mathcal{P}(\mu)/ \{$bounded sets$\}$.

\item For every $\alpha < \kappa^+$ and every $\mathbb P_\alpha$-name $\sigma$ for an antichain, there is $\beta \geq \alpha$ such that $\Vdash_\beta \dot{\mathbb Q}_\beta = \mathbb C_\sigma$.
\item Every $\mathbb P_{\kappa^+}$-name $\tau$ for an antichain, there is $\alpha < \kappa^+$ and a $\mathbb P_\alpha$-name $\sigma$ such that $\Vdash_{\kappa^+} \tau = \sigma$.
\end{enumerate}

An iteration satisfying (1) and (2) can be defined using a suitable bookkeeping function because inductively we have $\forces_{\alpha} 2^\kappa=\kappa^+$. Furthermore, (3) is a consequence of the fact that the entire iteration is $\kappa$-c.c.

Any ${<}\mu$-support iteration of ${<}\mu$-closed posets is ${<}\mu$-closed.  Therefore, there is a dense set of conditions $p$ such that at all $\alpha < \kappa^+$, there is $s \in V$ with $p \restriction \alpha \Vdash_\alpha p(\alpha) = (\check s,\dot T)$ for some name $\dot T$.  We may assume that we force with this dense suborder.

We show by induction that for all $\alpha < \kappa^+$, $|\mathbb P_\alpha| = \kappa$, $\mathbb P_\alpha$ is $\kappa$-c.c, and $\mathbb P_\alpha$ preserves $\GCH$.  The base case and successor steps are easy.  The cardinality claim at limit stages follows from the fact that $\kappa^{<\mu} = \kappa$.  To show the chain condition, let $\{ p_\beta : \beta < \kappa \} \subseteq \mathbb P_{\alpha}$ and let $A \in [\kappa]^\kappa$ be such that $\{ \supp p_\alpha : \alpha \in A \}$ is a $\Delta$-system and such that the bounded sets of $\mu$ mentioned on the root are all the same.  The chain condition and cardinality together imply that $\GCH$ is preserved going forward. In the end, $\Vdash_{\kappa^+} 2^\mu = \kappa^+$, but the $\kappa$-c.c.\ holds of the whole iteration for the same reason as above.

To get the desired consistency result, we use Corollary \ref{theorem_baumgartner_taylor}.  In our situation, if $j : V \to M \subseteq V[G]$ is a generic ultrapower embedding derived from a saturated normal ideal $I$ on $\kappa=\mu^+$ then it follows by elementarily that in $M$, the forcing $j(\mathbb P_{\kappa^+})$ is a ${<}\mu$-support iteration of $\mu$-centered forcings of length $j(\kappa^+)$.  Since $M^\mu \cap V[G] \subseteq M$, this holds in $V[G]$ as well.  Thus in $V[G]$, we can carry out the same $\Delta$-system argument to show that $j(\mathbb P_{\kappa^+})$ has the $\kappa^+$-c.c.  It follows from Theorem \ref{theorem_baumgartner_taylor} that $\bar I$, the ideal generated by $I$, is saturated $V^{\mathbb P_{\kappa^+}}$.

Let $H \subseteq \mathbb P_{\kappa^+}$ be generic.  To show $\bar I$ is $\mu$-minimal in $V[H]$, suppose $\bar G \subseteq \mathcal{P}(\kappa)/ \bar I$ is generic over $V[H]$, and $x \subseteq \mu$ is in $V[H][\bar G] \setminus V[H]$.  Let $\tau$ be a name for $x$. By \cite[Proposition 2.12]{MR2768692}, there is a function $f : \kappa \to \mathcal{P}(\mu)$ in $V[H]$ such that $\Vdash_{\bar I} [\check f]_{\dot G} = \tau$.  Since $\bar j(f)(\kappa) = x$ and $x \not= f(\alpha)$ for any $\alpha<\kappa$, $f$ is one-to-one on a $\bar I$-measure-one set, which we may assume is all of $\kappa$ by adjusting $f$ off this large set.  Since $2^{<\mu} = \mu$ in $V[H]$, each $f(\alpha)$ is coded by a branch through the complete binary tree of height $\mu$, and so the range of $f$ corresponds to a collection of almost-disjoint subsets of this tree.  Since $V[H]$ is a forcing extension by $\P_{\kappa^+}$, it follows that for any $Y \subseteq \kappa$ in $V[H]$, there is $y \subseteq \mu$ such that $y \cap f(\alpha)$ is unbounded in $\mu$ if and only if $\alpha \in Y$.  Thus we have $Y \in \bar G$ iff $\kappa \in \bar j(Y)$ iff $y \cap \bar j(f)(\kappa)$ is unbounded in $\mu$.  This implies that from $x$ we may recover $G$.

Now we show that $\bar I$ is rigid, following \cite{MR1955241}.

\begin{lemma}
\label{disjoint}
If $f,g$ are one-to-one functions with respective disjoint domains $A,B$ contained in a regular cardinal $\kappa$, then there are nonstationary $A',B'$ such that $f[A\setminus A'] \cap g[B \setminus B'] = \emptyset$.
\end{lemma}

\begin{proof}
Let $C= f[A] \cap g[B]$, and let $\pi : f^{-1}[C] \to g^{-1}[C]$ be a bijection such that $f(\alpha) = g(\pi(\alpha))$ for all $\alpha \in f^{-1}[C]$.  Neither $\pi$ nor $\pi^{-1}$ can be regressive on a stationary set.  Let $A' = \{ \alpha : \pi(\alpha) < \alpha \}$ and $B' = \{ \beta : \pi^{-1}(\beta) < \beta \}$.  If $x = f(\alpha) = g(\beta)$, then $\beta = \pi(\alpha)$.  Either $\beta < \alpha \in A'$, or $\alpha < \beta \in B'$.
\end{proof}

If $\bar I$ is nonrigid then whenever $G_0 \subseteq \mathcal{P}(\kappa)/\bar I$ is generic, there is a different generic $G_1$ such that $V[G_0] = V[G_1]$.  The corresponding generic ultrapower embeddings $j_0 : V[H] \to M_0=V[H]^\kappa/G_0$ and $j_1 : V[H] \to M_1=V[H]^\kappa/G_1$ are distinct, yet $\mathcal{P}(\mu)^{M_0} = \mathcal{P}(\mu)^{M_1} = \mathcal{P}(\mu)^{V[H][G_0]}$.  Let $x \in \mathcal{P}(\mu)^{V[H][G_0]} \setminus \mathcal{P}(\mu)^{V[H]}$.  Let $f,g$ be such that $j_0(f)(\kappa) = j_1(g)(\kappa) = x$.  Since $G_0 \not= G_1$, we can pick disjoint $A,B \subseteq \kappa$ such that $\kappa \in j_0(A)$ and $\kappa \in j_1(B)$.  As before, since $x \notin V[H]$, we may assume $f$ and $g$ are one-to-one on $A$ and $B$.

Using Lemma~\ref{disjoint}, we may also assume $f[A] \cap g[B] = \emptyset$.  Since $H$ is $\mathbb P_{\kappa^+}$-generic over $V$, there is $y \subseteq \mu$ such that $|y \cap z| = \mu$ for each $z \in f[A]$ and $|y \cap z| < \mu$ for each $z \in g[B]$.  By elementarity, $M_0 \models |y \cap x| = \mu$ and $M_1 \models |y \cap x| < \mu$.  But this is an absolute property between the models, so we have a contradiction.
\end{proof}

\section{Rigidity with $\GCH$}\label{section_rigidity_with_gch}

Suppose $\kappa$ is an inaccessible cardinal and $\mu<\kappa$ is regular. All of the standard posets used to force $\kappa=\mu^+$, such as the Levy collapse, Silver collapse, etc., have many nontrivial automorphisms. Hence, if $\P$ is one of these standard collapse forcings and $G$ is generic for $\P$ over $V$, then in $V[G]$ there are many distinct $V$-generic filters for $\P$. We will show that there is a forcing $\mathbb{C}$ such that if $G$ is generic for $\mathbb{C}$ over $V$, then $V[G]\models\kappa=\mu^+$ and, in $V[G]$, there is a unique $V$-generic filter for $\mathbb{C}$. We will use a variation of the coding forcing introduced by Friedman and Magidor \cite{MR2548481} to add a club which will code the generic for a collapsing forcing, as well as the generic for the coding forcing itself into the stationarity of subsets of $\kappa$.

Suppose $\P$ is some ${<}\mu$-closed forcing such that $\P$ is $\kappa$-c.c., $|\P|=\kappa$ and $\forced_\P \kappa=\mu^+$. Fix a bijection $b:\kappa\to\P$ and let $G$ be generic for $\P$ over $V$. Working in $V$, let $W,X,Y$ and $Z$ be increasing functions from $\kappa$ to $\kappa$ such that the ranges of $W,X,Y$ and $Z$ are each cofinal in $\kappa$ and together form a disjoint partition of $\kappa$. Let $\<\eta_\alpha\st\alpha<\kappa\>$ be an increasing enumeration of the regular cardinals in the interval $[\mu^+,\kappa)$. For each $\alpha<\kappa$ let $E_\alpha=\cof(\eta_\alpha)^V\cap[\eta_\alpha,\kappa)$ and let $\vec{E}_\kappa=\<E_\alpha\st\alpha<\kappa\>$. Notice that in $V[G]$ each set in the sequence $\vec{E}_\kappa$ remains stationary since $\P$ is $\kappa$-c.c. Working in $V[G]$, let $\Q=\Code(G,\vec{E}_\kappa)$ be the set of all closed bounded $c\subseteq\kappa=(\mu^+)^{V[G]}$ such that for $i<\kappa$,
\begin{itemize}
\item if $b(i)\in G$ then $c\cap E_{W(i)}=\emptyset$ and
\item if $b(i)\notin G$ then $c\cap E_{X(i)}=\emptyset$.
\end{itemize}
Conditions in $\Q$ are ordered by setting $d\leq c$ iff: 
\begin{enumerate}
\item $d$ is an end extension of $c$ and
\item for $i\leq\max(c)$, if $i\in c$ then $(d\setminus c)\cap E_{Y(i)}=\emptyset$ and if $i\notin c$ then $(d\setminus c)\cap E_{Z(i)}=\emptyset$.
\end{enumerate}
This defines the coding poset $\Code(G,\vec{E}_\kappa)\in V[G]$.

\begin{lemma}\label{lemma_code_distributive}
The poset $\Code(G,\vec{E}_\kappa)$ defined above is ${<}\mu$-closed and ${<}\kappa$-distributive in $V[G]$.
\end{lemma}

\begin{proof}

First we show that $\Code(G,\vec{E}_\kappa)$ is ${<}\mu$-closed in $V[G]$. Suppose $\gamma<\mu$ and $\<c_i\st i<\gamma\>\in V[G]$ is a decreasing sequence of conditions in $\Code(G,\vec{E}_\kappa)$. Let $\delta=\sup\{\max(c_i)\st i<\gamma\}$. Then $\cf(\delta)^{V[G]}\leq\gamma$, which implies $\cf(\delta)^V\leq\gamma<\mu$ and since every element of $\bigcup_{\alpha<\kappa} E_\alpha$ has cofinality greater than $\mu$ in $V$, it follows that $\bigcup_{i<\gamma}c_i\cup\{\delta\}\in\Code(G,\vec{E}_\kappa)$ is a lower bound of the sequence.

Next we show that $\Code(G,\vec{E}_\kappa)$ is ${<}\kappa$-distributive in $V[G]$. Since $\kappa=(\mu^+)^{V[G]}$, it will suffice to show that $\Code(G,\vec{E}_\kappa)$ is ${\leq}\mu$-distributive in $V[G]$. Fix a sequence $\vec{D}=\<D_i\st i<\mu\>$ of open dense subsets of $\Code(G,\vec{E}_\kappa)$ in $V[G]$ and a condition $c\in\Code(G,\vec{E}_\kappa)$. Let $S^\kappa_\mu=(\cof(\mu)\cap\kappa)^V$ and notice that $S^\kappa_\mu$ does not appear on the sequence $\vec{E}_\kappa=\<E_\alpha\st\alpha<\kappa\>$. Since $\P$ is $\kappa$-c.c., it follows that $S^\kappa_\mu$ is a stationary subset of $\kappa$ in $V[G]$. Thus, working in $V[G]$ we may fix some large regular cardinal $\theta$ and a well-ordering $<_\theta$ of $H_\theta$ and an elementary submodel $N\elemsub (H_\theta,\in,<_\theta)$ such that
\begin{itemize}
\item $c,\Code(G,\vec{E}_\kappa),S^\kappa_\mu,\vec{D}\in N$
\item $|N|^{V[G]}=\mu$
\item $N\cap\kappa\in S^\kappa_\mu$
\item $N^{<\mu}\cap V[G]\subseteq N$
\end{itemize}
Working in $V[G]$, let $\<\beta_i\st i<\mu\>$ be an increasing, continuous and cofinal sequence of ordinals in $\delta=N\cap\kappa$. Using the well-ordering $<_\theta$ and elementarity, we may build a decreasing sequence of conditions $\<c_i\st i<\mu\>$ such that $c_0=c$ and for each $i<\mu$ we have $(1)$ $c_{i+1}\leq c_i$, $(2)$ $c_{i+1}\in D_i$, $(3)$ $\beta_i\leq\max(c_i)$ and $(4)$ $c_i\in N$. At limit stages we make use of the fact that $\Code(G,\vec{E}_\kappa)$ is ${<}\mu$-closed in $V[G]$ and $N^{<\mu}\cap V[G]\subseteq N$. Since the ordinal $\delta= N\cap(\mu^+)^{V[G]}=\sup\{\max(c_i)\st i<\mu\}$ has cofinality $\mu$ in $V$, it follows that $\delta\notin\bigcup_{\alpha<\kappa} E_\alpha$, and thus $c_{\infty}=\bigcup_{i<\mu} c_i\cup\{\delta\}\in \Code(G,\vec{E}_\kappa)$ is a lower bound of the sequence $\<c_i\st i<\mu\>$.\end{proof}

\begin{lemma}\label{lemma_unique_generic}
Suppose $\kappa$ is an inaccessible cardinal and $\mu<\kappa$ is regular with $\mu^{<\mu}=\mu$. Let $\P$ be a forcing notion such that $b:\kappa\to \P$ is a bijection and $\forces_\P\kappa=\mu^+$. Suppose $G*H\subseteq\P*\Code(G,\vec{E}_\kappa)$ is generic over $V$ and let $C=\bigcup H$. Then in $V[G*H]$, we have
\begin{enumerate}
\item For $i<\kappa$, $b(i)\in G$ iff $E_{W(i)}$ is nonstationary and $b(i)\notin G$ iff $E_{X(i)}$ is nonstationary.
\item For $i<\kappa$, $i\in C$ iff $E_{Y(i)}$ is nonstationary and $i\notin C$ iff $E_{Z(i)}$ is nonstationary.
\item There is a unique $V$-generic filter for $\P*\Code(G,\vec{E}_\kappa)$ (in $V[G*H]$).
\end{enumerate}
\end{lemma}

\begin{proof}
The proof is similar to \cite[Lemma 8]{MR2548481}. 

(1) It follows from the definition of extension in $\Code(G,\vec{E}_\kappa)$ that if $b(i)\in G$ then $E_{W(i)}$ is nonstationary. Conversely, we will prove that if $b(i)\notin G$ then $E_{W(i)}$ remains stationary in $V[G*H]$. Suppose $b(i)\notin G$ and that $c\forces \dot{D}\subseteq\kappa$ is club. It will suffice to find an extension $c_\infty\leq c$ with $c_\infty\forces \dot{D}\cap E_{W(i)}\neq\emptyset$. Working in $V[G]$, since $E_{W(i)}$ is a stationary subset of $\kappa=(\mu^+)^{V[G]}$, it follows that for some large enough regular cardinal $\theta$, we may let $<_\theta$ be a well-order of $H_\theta$ and find $N\elemsub (H_\theta,\in,<_\theta)$ such that 
\begin{itemize}
\item $c, \Code(G,\vec{E}_\kappa), E_{W(i)},\dot{D}\in N$
\item $|N|^{V[G]}=\mu$
\item $N\cap \kappa \in E_{W(i)}$ 
\item $N^{<\mu}\cap V[G]\subseteq N$
\end{itemize}
We have $E_{W(i)}\subseteq\kappa\cap\cof(\mu)^{V[G]}$ and thus, working in $V[G]$, we may fix a sequence of ordinals $\<\beta_i\st i<\mu\>$ which is increasing, continuous and cofinal in $\delta=N\cap\kappa$. Using the well-order $<_\theta$ and elementarity, we recusively construct a decreasing sequence of conditions $\<c_i\st i<\mu\>$ and a sequence of ordinals $\<\eta_i\st i<\mu\>$ such that $c_0=c$ and for each $i<\mu$ we have (1) $c_{i+1}\forces\eta_{i+1}\in \dot{D}$, (2) $\beta_i\leq\max(c_i),\eta_{i+1}$, (3) $\eta_i<\eta_{i+1}$ and (4) $c_i\in N$. At limit stages we make use of the facts that $\Code(G,\vec{E}_\kappa)$ is ${<}\mu$-closed in $V[G]$ and $N^{<\mu}\cap V[G]\subseteq N$. Thus $\delta=N\cap\kappa=\sup\{\max(c_i)\st i<\mu\}=\sup\{\eta_i\st i<\mu\}$. Let $c_\infty=\bigcup\{c_i\st i<\omega\}\cup\{\delta\}$. Since $\delta\in E_{W(i)}$ and $b(i)\notin G$, it follows that $c_\infty$ is a condition in $\Code(G,\vec{E}_\kappa)$ and that $c_\infty$ extends each $c_i$. Hence $c_\infty\forces\delta\in \dot{D}\cap E_{W(i)}$. This completes the proof of (1).

(2) is similar to (1). 

For (3), suppose that in $V[G*H]$ there is a $V$-generic filter for $\P*\Code(G,\vec{E}_\kappa)$, call it $G'*H'$. Then $V\subseteq V[G'*H']\subseteq V[G*H]$. Suppose $G'\neq G$. Without loss of generality, suppose that for some $i<\kappa$ we have $b(i)\in G'\setminus G$. Then by (1), it follows that $E_{W(i)}$ is nonstationary in $V[G'*H']$, but becomes stationary in $V[G*H]$, which is impossible. The rest of the cases for (3) are similar.
\end{proof}

With the above technique of coding a generic for a collapse forcing, we are ready to prove Theorem \ref{theorem_rigid_presaturated}; that is, we will show that if $\kappa$ is an almost-huge cardinal and $\mu<\kappa$ is regular, then there is a forcing extension in which there is a rigid presaturated ideal on $\mu^+$ and $\GCH$ holds.

\begin{proof}[Proof of Theorem \ref{theorem_rigid_presaturated}]
Suppose $j:V\to M$ is an elementary embedding with critical point $\kappa$ such that $\lambda=j(\kappa)$, $M^{<\lambda}\subseteq M$ and $j$ is the ultrapower by an almost-huge tower (see \cite{MR3343538}). Without loss of generality, assume $\GCH$ holds. Let $\mu<\kappa$ be a regular cardinal and let $\P=\P_\kappa$ be Magidor's variation of Kunen's universal collapse for forcing $\kappa=\mu^+$, as given in Definition \ref{definition_kunen_iteration} above. Let $\dot{\Q}=\dot{\Col}(\kappa,<\lambda)$ be a $\P$-name for the Levy-collapse below $\lambda$. Assume $G*H$ is generic for $\P*\dot{\Q}$ over $V$. Working in $V$, let $\vec{E}_\kappa$ be the sequence of stationary subsets of $\kappa$ in the definition of the coding forcing above. It follows that each set in the sequence $\vec{E}_\kappa$ remains stationary in $V[G*H]$ and the poset $\Code(G,\vec{E}_\kappa)$ is the same whether defined in $V[G]$ or $V[G*H]$. We will prove that if $K$ is generic for $\Code(G,\vec{E}_\kappa)$ over $V[G*H]$, then in $V[G*(H\times K)]$ there is a rigid presaturated ideal on $\mu^+$.

First we argue that the embedding $j$ can be generically extended to have domain $V[G*(H\times K)]$. Since the poset $\Code(G,\vec{E}_\kappa)$ is ${<}\mu$-closed and has size $\kappa$ there is a regular embedding $\Code(G,\vec{E}_\kappa)\to \Col(\mu,\kappa)$ and by property (\ref{absorb_collapse}) of the universal collapse $\P$ listed above, there is a regular embedding $\Col(\mu,\kappa)\to j(\P)/G*H$. Thus we may let $\hat{G}$ be generic for the quotient $j(\P)/(G*(H\times K))$ and extend the elementary embedding to $j:V[G]\to M[\hat{G}]$. By elementarity, $j(\P)$ is $\lambda$-c.c.\ and thus $M[\hat{G}]^{<\lambda}\cap V[\hat{G}]\subseteq M[\hat{G}]$.

Next we must lift $j$ to have domain $V[G*H]$. In $M[\hat{G}]$, for each $\alpha<\lambda$, $j[H\restrict\alpha]$ is a directed subset of $\Col(\lambda,{<}j(\lambda))$ of size $<\lambda$, thus $m_\alpha=\inf j[H\restrict\alpha]$ is a condition in $\Col(\lambda,{<}j(\lambda))$. In $V[\hat{G}]$, let $\<D_\alpha\st\alpha<\lambda\>$ enumerate the dense open subsets of $\Col(\lambda,{<}j(\lambda))$ that live in $M[\hat{G}]$. We inductively build a chain $\<q_\alpha\st\alpha<\lambda\>$ such that \[\hat{H}=\{q\in\Col(\lambda,{<}j(\lambda))\st(\exists\alpha<\lambda) q_\alpha\leq q\}\] is an $M[\hat{G}]$-generic filter and $m_\alpha\in \hat{H}$ for all $\alpha<\lambda$. Assume that we have constructed a sequence of conditions $\<q_i\st i<\alpha\>$ such that for each $i<\alpha$, $q_i\leq m_i$, $q_i$ is compatible with all $m_\beta$ and $q_i\in D_i$. Let $q_\alpha'=\inf\{q_i\st i<\alpha\}$. Then $q_\alpha'$ is compatible with all $m_\beta$. Let $\gamma<j(\lambda)$ be such that $D_\alpha\restrict\gamma=_{\text{def}}D_\alpha\cap\Col(\lambda,{<}\gamma)$ is dense in $\Col(\lambda,{<}\gamma)$ and $q_\alpha'\in D_\alpha\restrict\gamma$. Choose $q_\alpha\in D_\alpha\restrict\gamma$ below $q_\alpha'\land m_\gamma$. Then $q_\alpha$ is compatible with all $m_\beta$ since if $\beta\geq\gamma$, $m_\beta\restrict\gamma=m_\gamma$. This completes the construction of $\hat{H}$. Since $j[H]\subseteq\hat{H}$ the embedding extends to $j:V[G*H]\to M[\hat{G}*\hat{H}]$.

Now let $c^*=\bigcup j[K]\cup\{\kappa\}$. We will check that $c^*$ is a condition in $j(\Code(G,\vec{E}_\kappa))=\Code(j(G),j(\vec{E}_\kappa))^{M[\hat{G}]}=\Code(j(G),j(\vec{E}_\kappa))^{V[\hat{G}]}$ extending every element of $j[K]=K$. It suffices to show that $\kappa$ is not in any of the stationary sets on the sequence $j(\vec{E}_\kappa)=\<\bar{E}_\alpha\st \alpha<j(\kappa)\>$ where $\bar{E}_\alpha=(\cof(\eta_\alpha)\cap [\eta_\alpha^+,j(\kappa)))^V$. If $\alpha\neq\kappa$ then clearly $\kappa\notin \bar{E}_\alpha$. If $\alpha=\kappa$ then $\kappa\notin \bar{E}_\alpha=\bar{E}_\kappa=(\cof(\kappa)\cap[\kappa^+,j(\kappa)))^V$. Thus $c^*$ is a master condition. Let $\hat{K}$ be generic for $j(\Code(G,\vec{E}_\kappa))$ over $M[\hat{G}*\hat{H}]$ with $c^*\in\hat{K}$. Then $j$ lifts to $j:V[G*(H\times K)]\to M[\hat{G}*(\hat{H}\times\hat{K})]$.

Let $U$ be the ultrafilter on $\mathcal{P}(\kappa)^{V[G*(H\times K)]}$ induced by this extended embedding: $U=\{X\in\mathcal{P}(\kappa)^{V[G*(H\times K)]}\st \kappa \in j(X)\}$. Consider the natural commutative diagram
\[
\begin{tikzcd}
V[G*(H\times K)] \arrow{dr}{j_U} \arrow{r}{j} & M[\hat{G}*(\hat{H}\times\hat{K})]\\
 & V[G*(H\times K)]^\kappa/U \arrow{u}{k}
\end{tikzcd}
\]
where $j_U$ is the ultrapower embedding and $k([f]_U)=j(f)(\kappa)$. Since $k$ is an elementary embedding, it follows that the ultrapower $V[G*(H\times K)]^\kappa/U$ is well-founded and can thus be identified with its transitive collapse $N$.  Note that $\crit k > \kappa$.

Let us now show that $N=M[\hat{G}*(\hat{H}\times\hat{K})]$. Recall that the original embedding $j:V\to M$ is the ultrapower by an almost-huge tower, and thus $M$ is the direct limit of a directed system of $\alpha$-supercompactness embeddings $j_\alpha:V\to M_\alpha$ for $\alpha<\lambda$. Every member of $M_\alpha$ is of the form $j_\alpha(f)(j_\alpha[\alpha])$ for some function $f\in V$ with $\dom(f)=[\alpha]^{<\kappa}$. If $k_\alpha:M_\alpha\to M$ is the factor map such that $j=k_\alpha\circ j_\alpha$, then the critical point of $k_\alpha$ is above $\alpha$, so $k_\alpha(x)=k_\alpha[x]$ whenever $M_\alpha\models|x|\leq\alpha$. Since $M$ is the direct limit of this system of supercompactness embeddings it follows that for all $x\in M$ there is some $\alpha<\lambda$ and some function $f\in V$ such that 
\[x=k_\alpha([f])=k_\alpha(j_\alpha(f)(j_\alpha[\alpha]))=j(f)(k_\alpha(j_\alpha[\alpha]))=j(f)(j[\alpha]).\]
Let $\beta$ be any ordinal. There is some $\alpha$ with $\kappa\leq\alpha<\lambda$ and some $f\in V$ such that $\beta=j(f)(j[\alpha])$. Let $b':\kappa\to\alpha$ be a bijection in $V[G*(H\times K)]$. Then $\beta=j(f)(j(b')[\kappa])=k(j_U(f)(j_U(b')[\kappa]))$. Thus $\beta\in \range(k)$. This implies that $k$ does not have a critical point and thus $N=M[\hat{G}*(\hat{H}\times \hat{K})]$.

The forcing used to extend the embedding to have domain $V[G*(H\times K)]$ was $\R=_{\text{def}}j(\P)/(G*(H\times K))*(\Code(j(G),j(\vec{E}_\kappa))/c^*)$. We will use Foreman's Duality Theorem (Theorem \ref{theorem_duality_theorem} above) to show that $\R$ is forcing-equivalent to forcing with $\mathcal{P}(\kappa)/I$ where $I\in V[G*(H\times K)]$ is defined as
\[I=\{X\in\mathcal{P}(\kappa)^{V[G*(H\times K)]}\st 1 \forced_{\R} [\id]_{\dot{U}} \notin j(X)\}\]
where $\dot{U}$ is an $\R$-name for the ultrafilter on $\kappa$ derived from the generic embedding. We need to verify that the ultrapower $N$ satisfies the hypotheses of Foreman's theorem. 

Let us show that there is a regular embedding $e:\P*(\Col(\kappa,{<}\lambda)\times\Code(\dot{G},\vec{E}))\to j(\P)$ in the ultrapower $N$ of the form $j(\<e_{\alpha}\st\alpha<\kappa\>)(\kappa)$, where $\<e_{\alpha}\st\alpha<\kappa\>\in \break V[G*(H\times K)]$ is a sequence of regular embeddings. Fix an increasing sequence $\<\kappa_\alpha\st\alpha<\kappa\>$ of inaccessible cardinals which is cofinal in $\kappa$. By the absorption properties of the universal collapse $\P=\P_\kappa$, there exist regular embeddings $e_{\alpha}:\P_{\kappa_\alpha}*(\Col(\kappa_\alpha,{<}\kappa)\times\Code(\dot{G}_{\kappa_\alpha},\vec{E}_{\kappa_\alpha}))\to \P_\kappa$ where the forcing $\Col(\kappa_\alpha,{<}\kappa)\times\Code(\dot{G}_{\kappa_\alpha},\vec{E}_{\kappa_\alpha})$ is defined in $V^{\P_{\kappa_\alpha}}$ just as $\Col(\kappa,{<}\lambda)\times\Code(\kappa,\vec{E}_\kappa)$ was defined in $V^{\P_\kappa}$. It follows that $e$ is represented in the ultrapower $N$ as $j(\<e_{\alpha}\st\alpha<\kappa\>)(\kappa)$. Thus $N$ computes the quotient algebra from $\hat{G}=j(G)$ and $e$, and this is represented by a function with domain $\kappa$. For each $p\in j(\P)$, there is an ordinal $\alpha<\lambda$ and a function $f_p\in V$ such that $p=j(f_p)(j[\alpha])=j_\alpha(f_p)(j_\alpha[\alpha])$. Using bijections $b_0:\kappa\to ([\alpha]^{<\kappa})^V$ and $b_1:\kappa\to\alpha$ in $V[G*(H\times K)]$, we can build a function $f'_p:\kappa\to \P$ that represents $p$ in the ultrapower $N$.  Thus, the hypotheses of Foreman's Duality Theorem are met.

By Lemma \ref{lemma_unique_generic}, there is in $V[\hat{G}*(\hat{H}\times\hat{K})]$ a unique generic for $\P_\lambda*\Code(j(G),j(\vec{E}_\kappa))$. Since there is a dense embedding $\mathcal{P}(\kappa)/I\to j(\P)/(G*(H\times K))*\Code(j(G),j(\vec{E}_\kappa))/c^*$, $\mathcal{P}(\kappa)/I$ is rigid. Since $\P_\lambda*\Code(j(G),j(\vec{E}_\kappa))$ preserves $\kappa^+$, it follows that $I$ is presaturated. Thus, in $V[G*(H\times K)]$ there is a rigid presaturated ideal on $\kappa=\mu^+$ and $\GCH$ holds. This completes the proof of Theorem \ref{theorem_rigid_presaturated}.
\end{proof}

In order to prove Theorem \ref{theorem_rigid_saturated} we will make use of certain coherent $\diamondsuit_{\mu,\lambda}$-sequences added by forcing of the form $\Add(\mu)*\dot{\P}$ where $\Add(\mu)$ is the forcing to add a Cohen subset to a regular cardinal $\mu$ and $\dot{\P}$ is an $\Add(\mu)$-name for ${<}\mu$-closed forcing. Recall that $\diamond_{\mu,\lambda}$ holds if there is a sequence $\<a_z\st z\in [\lambda]^{<\mu}\>$ such that for every $X\subseteq\lambda$ the set $\{z\in [\lambda]^{<\mu}\st X\cap z=a_z\}$ is stationary.  We use a stronger principle we call $\blacklozenge_{\mu,\lambda}$, which states that there is a sequence $\<a_z\st z\in [\lambda]^{<\mu}\>$ such that for all $X \subseteq \lambda$, all club $C \subseteq [\lambda]^{<\mu}$, and all $\alpha < \mu$, there is a strictly $\subset$-increasing continuous sequence $\< z_i : i < \alpha \>$ contained in $C$ such that for all $z_i$, $X \cap z_i = a_{z_i}$.

\begin{lemma}\label{lemma_forcing_diamonds}
If $\mu$ is a regular cardinal and $\mu<\lambda$, then after forcing with $\Add(\mu)*\dot{\P}$ where $\dot{\P}$ is an $\Add(\mu)$-name for a ${<}\mu$-closed forcing poset, there exists a $\blacklozenge_{\mu,\lambda}$-sequence.
\end{lemma}

\begin{proof}
Let $g*G$ be generic for $\Add(\mu)*\dot{\P}$ over $V$. We may view $g$ as a function $g:\mu\to 2$ and we will identify each $z\in [\lambda]^{<\mu}$ with a function $\ot(z)\to\lambda$ enumerating its elements in increasing order; in other words, $z(\alpha)$ denotes the $\alpha$-th element of $z$ where $\alpha<\ot(z)$. For each $z\in[\lambda]^{<\mu}$ with $z\cap\mu\in \mu$, we define $a_z=\{z(\beta)\st \beta<\ot(z) \land g(z\cap\mu+\beta)=1\}$. We will show that 
\[\forces_{\Add(\mu)*\dot{\P}} (\forall X\subseteq \lambda) \{z\in[\lambda]^{<\mu}\st X\cap z=\dot{a}_z\}\text{ is stationary.}\]

Let $\dot{X}$ and $\dot{C}$ be $\Add(\mu)*\dot{\P}$-names and $(p_0,q_0)\in \Add(\mu)*\dot{\P}$ be such that 
\[(p_0,\dot{q}_0)\forces \text{$\dot{X}\subseteq\lambda$ $\land$ ($\dot{C}\subseteq [\lambda]^{<\mu}$ is club)}.\]
We inductively construct a decreasing sequence of conditions $\<(p_\alpha,\dot{q}_\alpha)\st \alpha<\mu\>$ below $(p_0,\dot{q}_0)$ and an increasing sequence $\<z_\alpha\st\alpha<\mu\>$ of elements of $[\lambda]^{<\mu}$ as follows.
\begin{enumerate}
\item For each $\alpha<\mu$ both $\dom(p_\alpha)$ and $z_\alpha\cap\mu$ are ordinals.
\item If $\alpha=\xi+1$ is a successor stage in the construction we choose $(p_{\xi+1},\dot{q}_{\xi+1})$ and $z_{\xi+1}$ such that 
\begin{enumerate}
\item $(p_{\xi+1},\dot{q}_{\xi+1})$ decides $\dot{X}\cap \check{z}_\xi$ and forces $z_{\xi+1}\in \dot{C}$,
\item $\dom(p_\xi)\subseteq z_{\xi+1}\cap\mu$ and $\ot(z_\xi) \cdot 2 \subseteq\dom(p_{\xi+1})$

\end{enumerate}
\item If $\alpha$ is a limit, then letting $r_\alpha=\bigcup_{\beta<\alpha}p_\beta$ and $z_\alpha=\bigcup_{\beta<\alpha} z_\beta$, it follows from $(2)(b)$ that $\dom(r_\alpha)=z_\alpha\cap\mu$. We choose $(p_\alpha,\dot{q}_\alpha)$ such that
\begin{enumerate}
\item $\dom(p_\alpha)=z_\alpha\cap\mu+\ot(z_\alpha)$
\item $p_\alpha\restrict(z_\alpha\cap\mu)=r_\alpha$ and $p_\alpha(z_\alpha\cap\mu+\beta)=1$ iff $q_\alpha\forces z_\alpha(\beta)\in\dot{X}$.
\end{enumerate}
\end{enumerate}
Let $\gamma<\mu$ be a limit ordinal and let $\<\delta_\beta\st \beta\leq\gamma\>$ enumerate the first $\gamma+1$ limit ordinals. Then it follows that $(p_{\delta_\gamma},\dot{q}_{\delta_\gamma})\forces$ $\{z_{\delta_\beta}\st\beta\leq\gamma\}\subseteq\dot{C}$ $\land$ $(\forall\beta<\gamma)$ $\dot{X}\cap z_{\delta_\beta}=\dot{a}_{z_{\delta_\beta}}$.\end{proof}

There is a natural map $\pi_{\lambda,\lambda'}:[\lambda]^{<\mu}\to [\lambda']^{<\mu}$ where $\lambda'<\lambda$ defined by $\pi_{\lambda,\lambda'}(z)=z\cap\lambda'$. We now show that the $\blacklozenge_{\mu,\lambda}$-sequences defined in the previous lemma are coherent with respect to the projection maps $\pi_{\lambda,\lambda'}$.

\begin{lemma}\label{lemma_coherence}
Suppose $\mu$ is a regular cardinal, $\dot{\P}$ is an $\Add(\mu)$-name for a ${<}\mu$-closed forcing poset and suppose $g*G\subseteq\Add(\mu)*\dot{\P}$ is generic over $V$. Working in $V[g*G]$, for each $\lambda>\mu$ let $\vec{a}_\lambda=\<a_z^\lambda\st z\in [\lambda]^{<\mu}\>$ denote the $\blacklozenge_{\mu,\lambda}$-sequence defined in the proof of Lemma \ref{lemma_forcing_diamonds} and for each $X\subseteq\lambda$, let $S_X=\{z\in[\lambda]^{<\mu}\st X\cap z=a_z^\lambda\}$ denote the set on which $X$ is anticipated by $\vec{a}_\lambda$. The $\blacklozenge_{\mu,\lambda}$-sequences $\vec{a}_\lambda$ are coherent in the sense that if $X\subseteq\lambda$ and $\lambda'<\lambda$, then we have $S_{X\cap\lambda'}= \pi_{\lambda,\lambda'}[S_X]$.
\end{lemma}
\begin{proof} The $\blacklozenge_{\kappa,\lambda}$-sequences $\vec{a}_\lambda=\<a_z^\lambda\st z\in[\lambda]^{<\mu}\>$ are defined in the proof of Lemma \ref{lemma_forcing_diamonds}, by letting $a_z^\lambda=\{z(\beta)\st \beta<\ot(z) \land g(z\cap\mu+\beta)=1\}$; in other words, $a_z^\lambda$ is the subset of $z$ obtained by looking at the Cohen generic $g$ restricted to $[z\cap\mu, z\cap\mu+\ot(z))$. Hence it follows that for $\lambda'<\lambda$ and $z\in[\lambda]^{<\mu}$ we have (1) $a_z^\lambda\cap\lambda' = a_{z\cap\lambda'}^{\lambda'}$ and (2) $a_{z\cap\lambda'}^\lambda=a_{z\cap\lambda'}^{\lambda'}$. Now let us prove that $S_{X\cap\lambda'}=\pi_{\lambda,\lambda'}[S_X]$.

($\supseteq$) If $z'=z\cap\lambda'$ for some $z\in[\lambda]^{<\mu}$ with $X\cap z=a_z^\lambda$, then we have $(X\cap\lambda')\cap z'= (X\cap z)\cap\lambda'=a_z^\lambda\cap\lambda'=a_{z'}^{\lambda'}$ (by (1)). Thus $z'\in S_{X\cap\lambda'}$.

($\subseteq$) Suppose $z'\in S_{X\cap\lambda'}$. Then $z'\in[\lambda]^{<\mu}$ and $(X\cap\lambda')\cap z' = a_{z'}^{\lambda'}=a_{z'}^\lambda$ (by (2)). Thus $z'\in \{z\cap\lambda'\st z\in S_X\}$.
\end{proof}

We are now ready to prove the special case of Theorem \ref{theorem_rigid_saturated} in which $\mu$ is not the successor of a singular cardinal.  That case requires a bit more care and will be dealt with afterwards.

\begin{theorem}\label{theorem_rigid_saturated_mu_not_succ_sing}
Suppose $\GCH$ and $\kappa$ is an almost-huge cardinal and $\mu<\kappa$ is an uncountable regular cardinal which is not the successor of a singular cardinal. Then there is a ${<}\mu$-distributive forcing extension in which there is a rigid saturated ideal on $\mu^+$ and $\GCH$ holds.
\end{theorem}

\begin{proof}

Suppose $\kappa$ is an almost-huge with target $\lambda$, and let $\mu < \kappa$ be uncountable and such that $\mu^{<\mu} = \mu$. We must show that there is a forcing extension in which $\GCH$ holds and there is a rigid saturated ideal on $\mu^{+}$.  We first let $g$ be $\Add(\mu)$-generic over the ground model $V_0$, and let $V = V_0[g]$.  In $V$, fix a bijection $b : \kappa  \to V_\kappa$.

If $\mathbb P$ is the ${<}\mu$-closed Kunen collapse of $\kappa$ to $\mu^+$, we have $j(\mathbb P) \cap V_\kappa = \mathbb P$.  By construction, $\mathbb P * \Col(\kappa,{<}j(\kappa)) \lhd j(\mathbb P)$.  By same arguments as for Theorem \ref{theorem_rigid_presaturated}, if $G * H \subseteq \mathbb P * \Col(\kappa,{<}j(\kappa))$, then in $V[G*H]$, $\kappa = \mu^+$, and there is a $\kappa^+$-saturated ideal on $\kappa$ with quotient isomorphic to $\mathcal B(j(\mathbb P)/(G*H))$.  As in the proof of Theorem \ref{theorem_rigid_presaturated}, a generic embedding for this ideal will always extend the ground model almost-huge embedding.

For $\delta \in [\mu,\kappa)$, let $\langle a^\delta_z : z \in [\delta]^{<\mu} \rangle$ be the $\blacklozenge_{\mu,\delta}$ sequence given by the above lemmas.  In a slight abuse of notation, put $S^\delta_\alpha = \{ z : a^\delta_z = z \cap \{  \alpha  \} \}$ for $\alpha,\delta < \kappa$.    
Note that if $\alpha \not= \beta$, then $S^\delta_\alpha \cap S^\delta_\beta$ only contains those $ z \in [\delta]^{<\mu}$ for which $z \cap \{ \alpha,\beta \} = \emptyset$.

We now define the forcing which codes $G$ into destroying/preserving the stationarity of the $S^\delta_\alpha$.  Working $V[G]$, since all cardinals in $[\mu,\kappa)$ are collapsed to $\mu$, for each $\delta \in [\mu,\kappa)$, choose a continuous, increasing, cofinal sequence $\vec{z}(\delta)=\langle z^\delta_i : i < \mu \rangle \subseteq [\delta]^{<\mu}$.  Let $\mathbb C(G)$ be the collection of partial functions $p : \kappa \to \mathcal P(\mu)$ with the following properties:
\begin{enumerate}
\item $|p| < \mu$.
\item For all $\alpha < \kappa$, $p(\alpha)$ is a closed bounded subset of $\mu \setminus 1$.
\item For all $\alpha< \kappa$, if $b(\alpha) \in G$, then $S^{\mu \cdot \alpha + \mu}_{\mu \cdot \alpha} \cap \{ z^{\mu \cdot \alpha + \mu}_i : i \in p(\alpha) \} = \emptyset$. 
\item For all $\alpha< \kappa$, if $b(\alpha) \notin G$, then $S^{\mu \cdot \alpha + \mu}_{\mu \cdot \alpha + 1} \cap \{ z^{\mu \cdot \alpha + \mu}_i : i \in p(\alpha) \} = \emptyset$. 
\end{enumerate}
We let $q \leq p$ when:
\begin{enumerate}
\item $\dom p \subseteq \dom q$.
\item For all $\alpha \in \dom p$,  $q(\alpha) \cap (\max p(\alpha)+1) = p(\alpha)$.
\item If $\beta \in p(\alpha)$, then $S^{\mu \cdot \alpha + \mu}_{\mu \cdot \alpha + 2 \cdot \beta} \cap \{ z^{\mu \cdot \alpha + \mu}_i : i \in q(\alpha) \setminus p(\alpha) \} = \emptyset$. 
\item If $\beta \in [1,\max p(\alpha)] \setminus p(\alpha)$, then $S^{\mu \cdot \alpha + \mu}_{\mu \cdot \alpha + 2 \cdot \beta + 1} \cap \{ z^{\mu \cdot \alpha + \mu}_i : i \in q(\alpha) \setminus p(\alpha) \} = \emptyset$. 
\end{enumerate}

Note that since $\mu\cdot\alpha+\beta<\mu\cdot\alpha+\mu$ for all $\beta<\mu$, it follows that $S^{\mu\cdot\alpha+\mu}_{\mu\cdot\alpha+\xi}\cap S^{\mu\cdot\alpha+\mu}_{\mu\cdot\alpha+\zeta}$ is nonstationary for all $\xi\neq\zeta$ less than $\mu$. A standard $\Delta$-system argument establishes that $\mathbb C(G)$ is $\kappa$-c.c., and since $| \mathbb C(G) | = \kappa$, $\mathbb C(G)$ preserves $\GCH$ for cardinals $\geq \mu$.  The following is the key technical claim towards showing that $\mathbb C(G)$ forces the existence of a rigid saturated ideal, and it also shows that cardinals and $\GCH$ are preserved below $\mu$.

\begin{nonGlobalClaim}
\label{satlem}
Let $K \subseteq \mathbb C(G)$ be generic over $V[G]$. 
\begin{enumerate}
\item For all $\alpha$, $C_\alpha = \bigcup_{p \in K} p(\alpha)$ is club in $\mu$.
\item $\mathbb C(G)$ is ${<}\mu$-distributive.
\item For all $\alpha < \kappa$, $b(\alpha) \in G \Leftrightarrow S^{\mu \cdot \alpha + \mu}_{\mu \cdot \alpha}$ is nonstationary $\Leftrightarrow S^{\mu \cdot \alpha + \mu}_{\mu \cdot \alpha + 1}$ is stationary.
\item For all $\alpha < \kappa$ and $\beta < \mu$, $\beta \in C_\alpha \setminus 1 \Leftrightarrow S^{\mu \cdot \alpha + \mu}_{\mu \cdot \alpha + 2 \cdot \beta} $ is nonstationary $\Leftrightarrow S^{\mu \cdot \alpha + \mu}_{\mu \cdot \alpha + 2 \cdot \beta +1}$  is stationary.
\end{enumerate}
\end{nonGlobalClaim}
\begin{proof}
For (1), let $p \in \mathbb C(G)$, $\alpha < \kappa$, and $\xi < \mu$ be arbitrary.  Let $\gamma < \mu$ be greater than $2 \cdot \beta + 1$ for all $\beta \in p(\alpha)$.  $[\mu \cdot \alpha + \mu]^{<\mu} \setminus \bigcup_{i<\gamma} S^{\mu \cdot \alpha + \mu}_i$ is stationary, so let $z^{\mu \cdot \alpha + \mu}_\zeta$ be in this set, where $\zeta > \left(\bigcup p(\alpha)\right)\cup\xi$.  Then $(p \setminus (\alpha,p(\alpha))) \cup \{ (\alpha, p(\alpha) \cup \{ \zeta \} ) \}$ is a condition stronger than $p$.  This shows that $C_\alpha$ is forced to be unbounded.  It is closed because all initial segments are closed.

We will show (2) and (3) with one construction.  Let $\nu<\mu$ be a regular cardinal, $\dot f$ a name for a function from $\nu$ to the ordinals,  $\dot C$ a name for a club in $\mu$, $p_0 \in \mathbb C(G)$, and $\xi < \kappa$.  Let $\theta$ be a sufficiently large regular cardinal, and let  $\langle M_\alpha : \alpha < \mu \rangle$ be a continuous increasing sequence of elementary submodels of $H_{\theta}$, each of size $<\mu$ and having transitive intersection with $\mu$, with $\{\nu,\mu,\dot f,\dot C,p_0,\mathbb C(G),\xi,\vec{Z}\} \in M_0$ where $\vec{Z}=\<z^\delta_i\st \delta<\kappa,i<\mu\>$, and such that for all $\alpha < \mu$, then $M_\alpha \in M_{\alpha+1}$ and $M_{\alpha+1}^{<\nu} \subseteq M_{\alpha+1}$.

Let $D \subseteq \mu$ be the club set of $\alpha$ where $M_\alpha \cap \mu = \alpha$.  Let $T$ be $S^{\mu \cdot \xi + \mu}_{\mu \cdot \xi + n}$, where we let $n = 1$ if $b(\xi) \in G$ and $n= 0$ if $b(\xi) \notin G$.  Since $\<a^{\mu\cdot\xi+\mu}_z\st z\in[\mu\cdot\xi+\mu]^{<\mu}\>$ is a $\blacklozenge_{\mu,\mu\cdot\xi+\mu}$-sequence, we may let $\langle \alpha_i : i \leq \nu \rangle$ be a continuous increasing sequence contained in $D$ such that $\{ z^{\mu \cdot \xi + \mu}_{\alpha_i} : i \leq \nu \} \subseteq T$.  Choose a descending chain of conditions below $p_0$ as follows:  Given $p_i \in M_{\alpha_{i+1}}$, let $p_{i+1} \in M_{\alpha_{i+1}}$ be such that $\alpha_ i \leq \max p_{i+1}(\gamma)$ for all $\gamma \in \dom(p_{i+1})$.  This is possible by the argument for (1) and elementarity.  Also, let $p_{i+1}$ decide $\dot f(i)$ and force $\beta_{i+1} \in \dot C$ for some $\beta_{i+1}\geq \alpha_i$.

If $i$ is a limit, let $p_i(\gamma) = \bigcup_{k<i} p_k(\gamma) \cup \{ \alpha_i \}$ for each $\gamma \in \bigcup_{k<i} \dom p_k$, and define $\beta_i = \sup_{k<i} \beta_k$.  To show this is a condition, let $\eta < \kappa$ and $\beta < \mu$ be in $M_{\alpha_i}$, and let $S' = S^{\mu \cdot \eta + \mu}_{\mu \cdot \eta + \beta}$.  Note that for all $\delta \in [\mu,\kappa)\cap M_{\alpha_i}$, since $M_{\alpha_i}$ knows $\langle z^\delta_k : k < \mu \rangle$ is club in $[\delta]^{<\mu}$, $z^\delta_{\alpha_i} = M_{\alpha_i} \cap \delta$.  So the $z^\delta_{\alpha_i}$ project to one another under the natural maps defined just before Lemma \ref{lemma_coherence}.  Suppose $\mu \cdot \xi + n \not= \mu \cdot \eta + \beta$.
\begin{itemize}
\item \underline{Case 1:} $\eta \leq \xi$.
Let $T' = S^{\mu \cdot \eta + \mu}_{\mu \cdot \xi + n}$. By Lemma \ref{lemma_coherence}, $T'=\pi_{\mu\cdot\xi+\mu,\mu\cdot\eta+\mu}[T]$.  Since $T' \cap S' \cap \{ z : \mu \cdot \eta + \beta \in z \} = \emptyset$, and $\mu \cdot \eta + \beta \in z^{\mu \cdot \eta + \mu}_{\alpha_i} \in T'$, we have $z^{\mu \cdot \eta + \mu}_{\alpha_i} \notin S'$.
\item \underline{Case 2:} $\eta > \xi$.
Let $T' = \pi_{\mu \cdot \eta + \mu,\mu \cdot \xi + \mu}^{-1}[T]$.  If $z^{\mu \cdot \eta + \mu}_{\alpha_i} \in T' \cap S'$, then applying Lemma \ref{lemma_coherence}, $z^{\mu \cdot \xi + \mu}_{\alpha_i} \in T \cap S^{\mu \cdot \xi + \mu}_{\mu \cdot \eta + \beta}$, which is false since $\mu \cdot \xi + n \in z^{\mu \cdot \xi + \mu}_{\alpha_i}.$
\end{itemize}
In either case, we may add $\alpha_i$ to the closed bounded set at coordinate $\eta$ without violating the requirements for being in $\mathbb C(G)$.
We have that $p_i \Vdash \beta_i = \sup_{k <i} \beta_k \in \dot C$.  The construction continues because for each limit $i < \nu$, $\langle p_k : k \leq i \rangle \in M_{\alpha_{i+1}}$.  In the end, $p_{\alpha_\nu}$ is a condition deciding $\dot f$ and forcing $\alpha_\nu \in \dot C$ and $z^{\mu \cdot \xi + \mu}_{\alpha_\nu} \in T$.

To show (5), let $0<\beta < \mu$, $\xi < \kappa$, $q_0 \in \mathbb C(G)$, and let $\dot C$ be a name for a club.  Take $q_1 \leq q_0$ such that $\max(q_1(\xi)) > \beta$.  Let $\delta = 2 \cdot \beta + 1$ if $\beta \in q_1(\xi)$, and $\delta = 2 \cdot \beta$ if $\beta \notin q_1(\xi)$.  Construct a sequence of models as before, and take the analogous club $D$.  Let $\langle \alpha_i : i \leq \omega \rangle$ be a continuous increasing sequence contained in $D$ such that $\{ z^{\mu \cdot \xi + \mu}_{\alpha_i} : i \leq \omega \} \subseteq S^{\mu \cdot \xi + \mu}_{\delta}$.  As above, we may choose a descending chain of conditions $q_0 \leq q_1 \leq q_2 \leq ... \leq q_\omega$ such that $q_\omega \Vdash \alpha_\omega \in \dot C$ and $z^{\mu \cdot \xi + \mu}_{\alpha_\omega} \in S^{\mu \cdot \xi + \mu}_{\delta}$.
\end{proof}

\begin{nonGlobalClaim}
If $G * K$ is $\mathbb P * \dot{\mathbb C(G)}$-generic, then there is no other $\mathbb P * \dot{\mathbb C(G)}$-generic filter in $V[G][K]$.
\end{nonGlobalClaim}
\begin{proof}
If $G' * K' \in V[G * K]$ is another $\mathbb P * \dot{\mathbb C}(\dot G)$-generic over $V$, then either $G \not= G'$ or $K \not= K'$.  In the first case, assume $b(\xi) \in G \triangle G'$.   Then for some $n < 2$, $S^{\mu \cdot \xi + \mu}_{\mu \cdot \xi + n}$ is nonstationary in $V[G' * K']$ but stationary in $V[G*K]$, which is impossible.
In the second case, if $G = G'$ but $K \not= K'$, let $(\alpha,\beta)$ be such that $\beta \in C_\alpha \triangle C'_\alpha$.  Then for some $n < 2$, $S^{\mu \cdot \alpha + \mu}_{\mu \cdot \alpha + 2 \cdot \beta + n}$ is nonstationary in $V[G * K']$, but stationary in $V[G * K]$, which again is impossible.
\end{proof}

The above claim gives us what we want, via the Duality Theorem.  By Theorem \ref{dualitynicecase}, if $\bar I$ denotes the ideal generated by $I$ in $V[G*H]^{\mathbb C(G)}$, then $\mathcal{B}( \mathbb{C}(G) * \dot{\mathcal{P}} (\kappa)/ \bar I) \cong \mathcal{B}( \mathcal P(\kappa)/I * j(\mathbb{C}(\dot G)))$.  Since we can carry out the $\Delta$-system argument for the $\kappa^+$-c.c.\ of  $j(\mathbb C(G))$ in $V[G*H]^{\mathcal P(\kappa)/I}$, $\bar I$ is forced to be saturated.  By applying the above claim to $j(\mathbb P * \dot{\mathbb C}(\dot G))$, we see that $\bar I$ is forced to be rigid, because 
\begin{align*}
j(\mathbb P * \dot{\mathbb C}(\dot G))	& \sim  j(\mathbb P) * j(\dot{\mathbb C}(\dot G) \\
						& \sim  \mathbb P * \Col(\kappa,{<}\lambda) * \frac{j(\mathbb P)}{\dot G * \dot H} * j(\dot{\mathbb C}(\dot G)) \\
						& \sim  \mathbb P * \Col(\kappa,{<}\lambda) * \mathcal P(\kappa)/I * j(\dot{\mathbb C}(\dot G))   \\
						& \sim  \mathbb P * \Col(\kappa,{<}\lambda) * \dot{\mathbb C}(\dot G) * \mathcal P(\kappa)/\bar{I}.  \\
\end{align*}
A nontrivial automorphism of $\mathcal P(\kappa)/ \bar I$ would give a forcing extension of $V$ with distinct generics for $j(\mathbb P * \dot{\mathbb C}(\dot G))$.
\end{proof}

Finally, we sketch the proof of Theorem \ref{theorem_rigid_precipitous}, which closely follows the above arguments.  Let $\kappa$ be measurable with normal measure $U$, and let $\mu < \kappa$ be regular.  Let $g \subseteq \Add(\mu)$ be generic over the ground model $V_0$, so that in $V = V_0[g]$, we have the coherent $\blacklozenge_{\mu,\delta}$-sequences, which are indestructible by ${<}\mu$-closed forcing.  Take a generic $G \subseteq \Col(\mu,{<}\kappa)$, and let $I$ be the ideal generated by the dual of $U$.  By the Duality Theorem, $\mathcal P(\kappa)/I$ is forcing-equivalent to $\Col(\mu,{<}j(\kappa))$.  We then force over this model with $\mathbb C(G)$.  The generated ideal will be rigid, and the key reason is the following:  If $H \subseteq \mathcal P(\kappa)/I$ is generic and $j : V[G] \to M \subseteq V[G][H]$ is the generic ultrapower embedding, then $j(\mathbb C(G))$ is the same whether defined in $M$ or $V[G][H]$, since it uses the same $\blacklozenge_{\mu,\delta}$-sequences.  Thus the stationarity of the relevant sets is absolute between $M$ and $V[G][H]$, even though $M$ is not $\mu$-closed.

\section{Near singular cardinals}\label{section_singular}

In this section, we show how to modify the previous arguments to obtain $\GCH$ along with rigid saturated ideals on double successors of singulars, which is the remaining case of Theorem \ref{theorem_rigid_saturated}.
We do not need to singularize any large cardinals, but only collapse our almost-huge $\kappa$ to be $\mu^+$, where $\mu$ is the successor of a singular in the ground model.  We then apply the same forcing $\mathbb C(G)$, but we must work harder to prove the version of Claim~\ref{satlem} without the assumption that $\nu^{{<}\nu} < \mu$ for $\nu < \mu$.  Our argument is based on the proof of Theorem 2 in \cite{MR0716625}.

Let $\mu = \nu^+$, and assume we have forced the coherent $\blacklozenge_{\mu,\delta}$-sequences as before.  Fix a function $f : \mu \times \mu \to \nu$ such that  $f(\alpha, \cdot) \restriction \alpha$ is an injection of $\alpha$ into $\nu$ for each $\alpha < \mu$.  Let $\xi < \kappa$, and let $T = S^{\mu \cdot \xi + \mu}_{\mu \cdot \xi + \beta}$ for some $\beta < \mu$.
Let $G$ be generic for the Kunen collapse of $\kappa$ to $\mu^+$.
Let $\dot C$ be a $\mathbb C(G)$-name for a club in $\mu$, and $p_0 \in \mathbb C(G)$.  Let $\sigma$ be a name for a function from some $\delta < \nu$ to the ordinals, and assume $\delta > \cf(\nu)$. Let $\vec{Z}$ be as in the proof of Theorem \ref{theorem_rigid_saturated_mu_not_succ_sing}. Let $\theta$ be a sufficiently large regular cardinal, and let  $\langle M_\alpha : \alpha < \mu \rangle$ be a continuous increasing sequence of elementary submodels of $(H_{\theta},\in,<_\theta)$ where $<_\theta$ is a well-order of $H_\theta$, each of size $\nu$ and having transitive intersection with $\mu$, such that $\{\nu,\mu, f, \sigma, \dot C,p_0,\mathbb C(G),\xi,\vec{Z}\} \in M_0$. Build the models such that for each $\alpha$, $M_\alpha \cap \mu \in \mu$ and $\langle M_\beta : \beta \leq \alpha \rangle \in M_{\alpha+1}$.  We assume $\GCH$ holds, which implies that $\mathcal P(\alpha) \subseteq M_0$ for all $\alpha < \nu$.

Let $D = \{ \alpha : M_\alpha \cap \mu = \alpha \}$, and let $A \subseteq D$ be a closed subset of order-type $\delta^+$ such that $\{ z^{\mu \cdot \xi + \mu}_i : i \in A \} \subseteq T$.  Fix a cofinal sequence $\langle \gamma_i : i < \cf(\nu) \rangle$ in $\nu$, where $\gamma_0 \geq \delta$.  Let $h : [A]^2 \to \cf(\nu)$ be defined as $h(\{\alpha,\beta\}) =$ the least $i$ such that $f(\alpha,\beta) < \gamma_i$, when $\alpha> \beta$.  Using the Erd\H{o}s-Rado Theorem, let $B \subseteq A$ have order-type $\delta$ and be homogeneous in the coloring $h$, say of color $\eta$.

Now we construct a descending chain $\langle p_i : i \leq \delta \rangle$ below $p_0$ as before. Let $\langle \alpha_i : i \leq \delta \rangle$ enumerate the closure of $B$, except that we skip the first successors of limits.
Given $p_i \in M_{\alpha_{i+1}}$, let $p_{i+1}$ be the $<_\theta$-least condition such that:
\begin{enumerate}
\item $\alpha_i \leq \max p_{i+1}(\gamma)$ for all $\gamma \in \dom(p_{i+1})$.
\item $p_{i+1}$ decides $\sigma(i)$ and some $\beta_{i+1} \in \dot C$ such that $\alpha_i \leq \beta_{i+1}$.
\end{enumerate}
At limit $i$, we define $\beta_i$ and $p_i$ as before.  The key thing to check is that at such a stage, $p_i \in M_{\alpha_{i+1}}$.  Since we skipped successors of limit points of $B$, there is some $\alpha^* \in M_{\alpha_{i+1}} \cap B \setminus \alpha_i$.
Since $f(\alpha^*,\cdot)[B \cap \alpha_i]$ is a subset of $\gamma_\eta$, $\langle \alpha_k : k \leq i \rangle \in M_{\alpha_{i+1}}$.  Thus the sequence $\langle p_k : k \leq i \rangle $
is definable from parameters in $M_{\alpha_{i+1}}$.

\section{The number of automorphisms}\label{section_questions}

%

In contrast to the application of this kind of coding method by Friedman and Magidor, we do not have such fine control over the number of automorphisms of boolean algebras.  For suppose $\mathbb B$ is a boolean algebra.  Let $\pi$ be a nontrivial automorphism of $\mathbb B$.  Let $a$ be such that $\pi(a) \not= a$.  Without loss of generality, $a \nleq \pi(a)$.  Let $b \leq a$ be such that $b \perp \pi(a)$.  Then $b \perp \pi(b)$.  Note that $\mathbb B {\restriction} b \cong\mathbb B {\restriction} \pi(b)$.
Now we define another automorphism $\sigma$.  Let $d \leq b$.  For $p \in \mathbb B$, define $\sigma(p) =$
\[ (p \wedge \neg(b \vee \pi(b))) \vee \pi(p \wedge d) \vee (p \wedge (b \setminus d)) \vee \pi^{-1}(p \wedge \pi(d)) \vee (p \wedge \pi(b \setminus d)). \]
The idea is simply to interchange two isomorphic cones below two elements of a five-element partition. Thus the number of automorphisms is at least $| \mathbb B {\restriction} b|$.  If $\mathbb B = \mathcal P(\kappa) / I$ for a normal ideal $I$ on a successor cardinal $\kappa$, then since $\mathbb B$ is nowhere $\kappa$-c.c., $| \mathbb B {\restriction} b | = 2^{\kappa}$.

More automorphisms might exist.  Suppose $2^{\omega_1} = \omega_2$ and $I$ is a normal ideal on $\omega_1$ such that $\mathcal P(\omega_1)/I \cong \mathcal B( \Col(\omega,\omega_1) \times \Add(\omega,\omega_2))$\footnote{This is consistent relative to infinitely many Woodin cardinals; see \cite{Woodin}.}.  Any two distinct functions from $\omega_2$ to $2$ induce distinct bit-flipping automorphisms of $\Add(\omega,\omega_2)$.  Thus there are $2^{\omega_2}$ many automorphisms, the maximum possible number.

\begin{question}
Can there exists a precipitous normal ideal $I$ on a successor cardinal $\kappa$ such that the number of automorphisms of $\mathcal P(\kappa) / I$ is greater than 1 and less than $2^{2^\kappa}$?
\end{question}


\end{document}